\documentclass[a4paper,12pt,reqno]{amsart}
 
\usepackage[margin=2cm]{geometry} 
\usepackage{amsmath,amsthm,amssymb,enumitem}
\usepackage{amsrefs, dsfont}
\usepackage[ocgcolorlinks,hyperfootnotes=false,colorlinks=true,citecolor=blue,linkcolor=blue,urlcolor=blue]{hyperref}

\newcommand{\fl}{f{\kern0.075em}l}

\theoremstyle{plain}
\newtheorem{thm}{Theorem}[section]
\newtheorem{lem}[thm]{Lemma}
\newtheorem{prop}[thm]{Proposition}
\newtheorem{cor}[thm]{Corollary}

\theoremstyle{definition}
\newtheorem{defn}[thm]{Definition}
\newtheorem{exmp}[thm]{Example}
\newtheorem{Remark}[thm]{Remark}

\newcommand{\mbb}{\mathbb}

\newcommand{\ti}{\tilde}

\usepackage{orcidlink}

\setlist[enumerate]{itemsep=2mm, topsep=0mm}

\title[Strong Localization of the Kobayashi--Eisenman Volume Element
]{Strong Localization of the Kobayashi--Eisenman Volume Element and Its Boundary Asymptotics}
\author{Ravi Shankar Jaiswal\,\orcidlink{0009-0003-5845-0485}}
\address{Department of Mathematics, Southern University of Science and Technology, Xueyuan Avenue, Shenzhen, Guangdong, China 518055}
\email{ravi@sustech.edu.cn}
\author{Debaprasanna Kar\,\orcidlink{0000-0002-1101-6214}}
\address{Institute of Mathematics and Applications, Bhubaneswar, Odisha, India 751029}
\email{deba.ima@iomaorissa.ac.in}
\date{\today}

\begin{document}

\addtolength{\jot}{2mm}
\addtolength{\abovedisplayskip}{1mm}
\addtolength{\belowdisplayskip}{1mm}

\maketitle
\begin{abstract}
 We establish a quantitative version of strong localization of the Kobayashi–Eisenman volume element and the quotient invariant near plurisubharmonic peak points of domains in $\mathbb{C}^n$. As an application of this strong localization result, we derive the non-tangential asymptotic limit of the Kobayashi–Eisenman volume element at exponentially flat infinite type boundary points of domains in $\mathbb{C}^{n+1}$.
\end{abstract}

\section{Introduction}\label{Intro}

The study of complex-geometric objects and invariants has occupied a fundamental position in several complex variables and complex geometry since the beginning of the 20th century, with the pioneering works of Carath\'eodory, Kobayashi, Bergman, Cartan, and Eisenman. Among the most influential invariant objects are the Carath\'eodory and Kobayashi metrics, together with their higher-dimensional analogues, namely the Carath\'eodory and Kobayashi–Eisenman volume elements. These invariant volume forms provide powerful tools for understanding biholomorphic equivalence or non-equivalence of domains, analytic continuation, hyperbolicity, and the asymptotic geometry of domains in complex Euclidean spaces. The present paper is devoted to the strong localization properties of the Kobayashi–Eisenman volume element and the quotient invariant, and to the analysis of boundary behaviour of the Kobayashi–Eisenman volume element near exponentially flat boundary points.

We denote by $\mathbb{D}$ the unit disc in $\mathbb{C}$, by $\mathbb{B}^n$ and $\mathbb{D}^n$ the unit ball and the unit polydisc in $\mathbb{C}^n$ respectively, and by $\mathbb{B}^n(p,r)$ the ball in $\mathbb{C}^n$ with centre $p$ and radius $r$. Let $D\subset \mbb C^n$ be a domain. The Carath\'eodory–Eisenman volume element at a point $z\in D$ is defined by 
\[c_D(z) = \operatorname{sup}\left\{\big|\operatorname{det} \psi'(z)\big|^2: \psi \in \mathcal{O}(D, \mbb B^n), \psi(z) = 0\right\},\]
and the Kobayashi–Eisenman volume element at $z\in D$ is defined by
\[ k_D(z) = \operatorname{inf}\left\{\big|\operatorname{det}\psi'(0)\big|^{-2} : \psi \in \mathcal{O}(\mbb B^n, D), \psi(0) = z\right\}.\]
Using the Schwarz lemma, it is easy to see that the Carath\'eodory and Kobayashi–Eisenman volume elements satisfy the fundamental inequality
\begin{align}\label{schwarz1}
  c_D(z) \leq k_D(z).  
\end{align}
Moreover, under a holomorphic map $F: D_1\to D_2$, volume elements satisfy the following transformation rule
\begin{align}
    v_{D_1}(z) \geq \big|\operatorname{det}F'(z)\big|^2 v_{D_2}\big(F(z)\big),
\end{align}
where $v = c, k$. Consequently, the above inequality becomes equality if $F$ is a biholomorphism.

By Montel's theorem, the supremum defining $c_D(z)$ is always attained, and if $D$ is a taut, then the infimum defining $k_D(z)$ is also attained.  
If $k_D$ is non-vanishing (for example when $D$ is taut), then by the transformation rule,
\begin{align}
    q_D(z) = \frac{c_D(z)}{k_D(z)}
\end{align}
turns out to be a biholomorphic invariant and is called the \emph{quotient invariant}.

The inequality in (\ref{schwarz1}) readily implies that $q_D(z) \leq 1$ for any domain $D\subset \mbb C^n$.  It is a remarkable result that if $D$ is any domain in $\mbb{C}^n$ and $z_0$ is any point in $D$ such that $q_D(z_0)=1$, then $q_D \equiv 1$ and $D$ is biholomorphic to $\mbb{B}^n$. This result has been proved by many authors with several conditions assumed on the domain $D$ (cf. \cite{Wong} when $D$ is bounded and complete hyperbolic, \cite{Ro} when $D$ is only assumed to be bounded, \cite{Dektyarev} when $D$ is hyperbolic). Later, Graham and Wu \cite{Graham} proved this result for any domain $D$ without assuming any condition on it. Therefore, the quotient invariant is particularly useful in characterizing geometric properties of domains, especially in detecting biholomorphic equivalence to the unit ball.

Eisenman \cite{Eisenman} constructed intrinsic measures on complex manifolds by giving a sequence of volume forms, which were modelled on the construction of the Kobayashi metric. In the same monograph, Eisenman also introduced their dual measures, called the Carat\'eodory measures. The top-dimensional Eisenman volume elements, which we now call the Kobayashi–Eisenman and the Carath\'eodory-Eisenman volume elements, play crucial roles in measuring hyperbolicity in complex manifolds, as observed by Graham and Wu \cite{Graham and Wu}. The boundary behaviour of the Kobayashi–Eisenman volume element, along with its localization, were studied for strongly pseudoconvex domains in \cites{Daowei, Cheung}. Classical localization results were established near local holomorphic peak points of bounded pseudoconvex domains. Strong localization results, as we will see, can be derived near local plurisubharmonic peak points, even though the domain is unbounded. Fornæss and Nikolov \cite{Nikolai} derived strong localization results for Kobayashi, Azukawa and Sibony metrics. In the same paper, they also discussed localization related results for the squeezing function, another important invariant in the study of several complex variables. We now state our first result concerning the strong localization property of the Kobayashi–Eisenman volume element and the quotient invariant near a boundary point admitting a local plurisubharmonic peak function.

\begin{thm}\label{1.1}
Let $D \subset \mathbb{C}^n$ be a domain and let $p \in bD$. Suppose that there exists a local plurisubharmonic peak function $\phi$ at $p \in bD$. Then, for any neighbourhood $U$ of $p$, there exist a neighbourhood $V \subset U$ of $p$ and a constant $m > 0$ such that, for $z \in D \cap V$,
\begin{enumerate}[labelsep=10mm,leftmargin=20mm,itemsep=5mm]
    \item[\emph(i)] \quad \quad \quad \quad \quad \quad \quad \quad \quad \quad
    $\begin{aligned}[t]
        k_D(z) \geq e^{{m}\phi(z)} k_{D \cap U}(z), \,\, \text{and}
    \end{aligned}$
    \item[\emph(ii)]  \quad \quad \quad \quad \quad \quad \quad \quad \quad \quad
    $\begin{aligned}[t]
    q_{D \cap U}(z) \geq e^{{m}\phi(z)} q_D(z).
    \end{aligned}$
    \end{enumerate}
\end{thm}
\begin{Remark}
    If $D \subset \mathbb{C}^n$ is a bounded domain, then Theorem \ref{1.1} holds for any pair of neighbourhoods $V$ and $U$ of $p$ satisfying $V \subset \subset U$ and $D \cap U$ is connected.
\end{Remark}
We now introduce the necessary definitions and notation to state our next result.

Informally, an exponentially flat boundary point is a boundary point of a domain at which the defining function vanishes faster than any polynomial order. Such points are examples of boundary points of infinite type, in the sense of D'Angelo. A detailed discussion of exponentially flat boundary points is given in the next section. For $\alpha, N > 0$, let
\begin{align}
    C_{\alpha,N} := \Big\{\big(z_1,z'\big) \in \mathbb{C}\times\mathbb{C}^{n} : \operatorname{Re}z_1 < -\alpha|z'|^N\Big\},
\end{align}
where $z' = (z_2, \dots, z_{n + 1})$.
An \emph{$(\alpha, N)$-cone type stream} approaching $0 \in bD$ is a smooth curve $\gamma : (0, \epsilon_0) \to D \cap C_{\alpha, N}$ with
\begin{equation*}
    \lim_{t \to 0^+} \gamma(t) = (0, \dots, 0),
\end{equation*}
for some $\epsilon_0 > 0$. Here, we denote the Euclidean distance of $\gamma(t)$ to $bD$ along the normal and tangential directions by $d(t)$ and $d^*(t)$, respectively.

The boundary behaviour of Eisenman volume elements and the quotient invariant on strongly pseudoconvex domains had been studied by several authors, see for example \cites{Daowei, Green-Krantz, Krantz book}. Many also studied the boundary behaviour of these volume elements on several classes of weakly pseudoconvex domains. For example, Nikolov \cite{Nikolov 2018} demonstrated the boundary limits of volume elements on h-extendible domains, after localizing them near holomorphic peak points. Recently, Borah and Kar \cite{Borah-Kar} derived boundary asymptotics of volume elements on convex finite type and Levi corank one domains in terms of the sizes of McNeal's and Catlin's polydisc, respectively. We next present our result on the non-tangential boundary asymptotic of the Kobayashi–Eisenman volume element near an exponentially flat boundary point.
\begin{thm}\label{kob vol elem}
     Let $D \subset \mathbb{C}^{n + 1}$ be a domain and let $0 \in bD$. Assume $0$ is an exponentially flat boundary point and $\gamma$ is an $(\alpha, N)$-cone type stream approaching $0$. Then,
     \begin{align}
       \lim_{t \to 0^{+}} \frac{k_{D}\left(\gamma(t)\right)}{d(t)^{-2} d^{*}(t)^{-2n}} = \frac{1}{4}.
   \end{align}
\end{thm}
This article is organized as follows. In Section~\ref{Prelim}, we present the definitions, preliminary results, and known results that are essential for proving our main theorems. In Section~\ref{Section 3}, we establish the strong localization of the Kobayashi–Eisenman volume element and the quotient invariant. Finally, in Section~\ref{Cone}, we prove the asymptotic behaviour of the Kobayashi–Eisenman volume element using the strong localization method developed in Section~\ref{Section 3}, together with the scaling method.
\section{Preliminaries}\label{Prelim}
In this section, we first define exponentially flat infinite type domains and then prove, using the Hadamard inequality, that the Kobayashi–Eisenman volume element of the polydisc at the origin is equal to one. This result will be useful in obtaining the exact constant in the asymptotic limit of the Kobayashi–Eisenman volume element at exponentially flat boundary points; see Theorem \ref{kob vol elem}. We next introduce local plurisubharmonic peak and antipeak functions, discuss several examples, and recall a result of Gaussier \cite{Gaussier} (stated as Lemma \ref{ex. neg peak and antipeak}), which guarantees the existence of global plurisubharmonic peak and antipeak functions whenever local ones exist simultaneously. Finally, we define the pluricomplex Green function and recall the Schwarz lemma for log-subharmonic functions, both of which will play important roles in the proofs of our main results.

Let $D \subset \mathbb{C}^{n + 1}$ be a domain with $0 \in bD$. The boundary point
$0$ is said to be \emph{exponentially flat} if there exists a local defining function of $D$ near the origin of the form
\begin{align}\label{1}
    \rho(z_1, \dots, z_{n + 1}) = \operatorname{Re}z_1 + \phi \left(|z_2|^2 + \dots + |z_{n + 1}|^2\right),
\end{align}
 where $\phi : \mathbb{R} \to \mathbb{R}$ is a smooth function that is \emph{exponentially flat at the origin} (definition below).
 {
 \begin{defn}\label{exp flat fun}
    A smooth function $\phi: \mathbb{R} \to \mathbb{R}$ is said to be
    \emph{exponentially flat at the origin}, if it satisfies the following properties:
    \begin{enumerate}
    \item $\phi(x) = 0,$ for $x \leq 0$, 
    \item there exists $\epsilon_0 > 0$ such that $\phi''(x) > 0$, for $0 < x < \epsilon_0$, and
    \item the function 
    \begin{align*}
        \psi(x) := \begin{cases} 
      -1/ \operatorname{log}(\phi(x)), &  \text{if } 0 < x < \epsilon_0,\\
      0, & \text{if } x = 0,
    \end{cases}
    \end{align*}
    satisfies
    \begin{align}\label{asy of psi}
        \lim_{x \to 0^{+}}\frac{\psi(x)}{x^m} = C, \text{\,for some $m, C > 0$.}
    \end{align}
\end{enumerate}
\end{defn}
\begin{Remark}
If $\phi$ is exponentially flat at the origin, then,
\begin{align}\label{5}
    \phi^{(k)}(0) = 0, \quad \text{for all } k \in \mathbb{N}.
\end{align}
Using \eqref{5}, we can easily see that exponentially flat boundary points are points of infinite type in the sense of D'Angelo \cite{D'Angelo 1982}.
\end{Remark}}
\begin{exmp}
For $m \in \mathbb{R}^{+}$,
\begin{align*}
    \phi(x) = \begin{cases} 
      0,  &\text{if } x \leq 0,\\
      \operatorname{exp}\left(-1/x^m\right),  & \text{if } x > 0,
      \end{cases}
\end{align*}
is exponentially flat at the origin.
\end{exmp}
We will prove the non-tangential asymptotic behaviour of the Kobayashi–Eisenman volume element at exponentially flat boundary points of domains in Section \ref{Cone}.

We will now use the following classical Hadamard inequality to prove $k_{\mathbb{D}^n}(0) = 1$.

\begin{thm}[Hadamard inequality]\label{Hadamard's ineq}
Let $A=(a_{ij})_{i,j=1}^{n}$ be a complex matrix. Then,
\begin{align*}
    |\det A|
    \leq
    \prod_{i=1}^{n}
    \left(
    \sum_{j=1}^{n}|a_{ij}|^2
    \right)^{1/2}.
\end{align*}
\end{thm}

\begin{lem}\label{vol of polyballs}
The Kobayashi–Eisenman volume element on the polydisc satisfies
\begin{align}
    k_{\mathbb{D}^n}(0)=1.
\end{align}
\end{lem}

\begin{proof}
By definition,
\begin{align*}
    k_{\mathbb{D}^n}(0)
    =
    \inf
    \left\{
    |\det \psi'(0)|^{-2}
    :
    \psi \in \mathcal{O}(\mathbb{B}^n,\mathbb{D}^n),
    \ \psi(0)=0
    \right\}.
\end{align*}

Since the inclusion map $i:\mathbb{B}^n \hookrightarrow \mathbb{D}^n$ belongs to $\mathcal{O}(\mathbb{B}^n,\mathbb{D}^n)$, we immediately obtain
\begin{align*}
    k_{\mathbb{D}^n}(0)\leq 1.
\end{align*}

Thus, it suffices to prove that
\begin{align*}
    |\det \psi'(0)|\leq 1
\end{align*}
for every holomorphic map $\psi \in \mathcal{O}(\mathbb{B}^n, \mathbb{D}^n)$ satisfying $\psi(0)=0$.

For $\psi \in \mathcal{O}(\mathbb{B}^n, \mathbb{D}^n)$ with $\psi(0) = 0$, we have 
\[
\psi'(0)
=
\left(
\frac{\partial \psi_i}{\partial z_j}(0)
\right)_{1 \leq i,j \leq n}
\]
the complex Jacobian matrix of $\psi$ at the origin. By the distance-decreasing property of the Kobayashi metric,
\begin{align}\label{kob relat}
    F_{\mathbb{D}^n}^{K}\big(0,\psi'(0)v\big)
    \leq
    F_{\mathbb{B}^n}^{K}(0,v),
    \quad \text{for any}\; v \in \mathbb{C}^n.
\end{align}

Since $F_{\mathbb{D}^n}^{K}(0,w)
=
\max_{1 \leq i \leq n}|w_i|$,
and
$F_{\mathbb{B}^n}^{K}(0,v)=|v|,$ it follows from \eqref{kob relat} that
\begin{align}\label{ineq for jacob}
    \max
    \bigg\{
    \big|\langle \psi'(0)v,e_1\rangle \big|,
    \dots,
    \big|\langle \psi'(0)v,e_n\rangle \big|
    \bigg\}
    \leq |v|,
\end{align}
for any $v \in \mathbb{C}^n$.

Fix $i \in \{1,\dots,n\}$ and choose
\[
v=
\bigg(
\overline{
\frac{\partial \psi_i}{\partial z_1}(0)},
\dots,
\overline{
\frac{\partial \psi_i}{\partial z_n}(0)}
\bigg).
\]
Substituting this vector in \eqref{ineq for jacob}, we obtain
\begin{align*}
    \sum_{j=1}^{n}
    \left|
    \frac{\partial \psi_i}{\partial z_j}(0)
    \right|^2
    \leq
    \left(
    \sum_{j=1}^{n}
    \left|
    \frac{\partial \psi_i}{\partial z_j}(0)
    \right|^2
    \right)^{1/2}.
\end{align*}
Hence,
\begin{align}
    \sum_{j=1}^{n}
    \left|
    \frac{\partial \psi_i}{\partial z_j}(0)
    \right|^2
    \leq 1.
\end{align}

Applying Hadamard's inequality (Theorem~\ref{Hadamard's ineq}) yields
\begin{align*}
    |\det \psi'(0)|\leq 1.
\end{align*}
Therefore,
\begin{align*}
    k_{\mathbb{D}^n}(0)=1.
\end{align*}
\end{proof}

\begin{cor}
The quotient invariant on $\mathbb{D}^n$ satisfies
\begin{align*}
    q_{\mathbb{D}^n}(0)=n^{-n}.
\end{align*}
\end{cor}

\begin{proof}
    Carath\'eodory \cite{Car} proved that $c_{\mbb D^n}(0)=n^{-n}$. Hence the result now follows immediately from Lemma~\ref{vol of polyballs}.
\end{proof}

Using the above corollary, the biholomorphic invariance of $q$, and the fact that the polydisc is a homogeneous domain, one can compute $q_{\mbb D^n}(z)$ for all other points $z\in \mbb D^n$.

\begin{cor}\label{k for poly ball}
The Kobayashi–Eisenman volume element on $\mathbb{D}\times \mathbb{B}^n$ satisfies
\begin{align}
    k_{\mathbb{D}\times \mathbb{B}^n}(0)=1.
\end{align}
\end{cor}

\begin{proof}
Since
\[
\mathbb{B}^{n+1}
\subset
\mathbb{D}\times \mathbb{B}^n
\subset
\mathbb{D}^{n+1},
\]
the monotonicity of the Kobayashi–Eisenman volume element implies
\begin{align*}
    k_{\mathbb{B}^{n+1}}(0)
    \geq
    k_{\mathbb{D}\times \mathbb{B}^n}(0)
    \geq
    k_{\mathbb{D}^{n+1}}(0).
\end{align*}
By Lemma~\ref{vol of polyballs}, $k_{\mathbb{D}^{n+1}}(0)=1$,
and using the Schwarz lemma, it is easy to see that $k_{\mathbb{B}^{n+1}}(0)=1$. Therefore,
\[
k_{\mathbb{D}\times \mathbb{B}^n}(0)=1.
\]
\end{proof}
We now define local plurisubharmonic peak and antipeak functions.

\begin{defn}
Let $D \subset \mathbb{C}^n$ be a domain and let $p \in bD$. A function $\phi$ is called a \emph{local plurisubharmonic peak function} at $p$ if there exists a neighbourhood $U$ of $p$ such that:
\begin{enumerate}
    \item $\phi$ is plurisubharmonic on $D \cap U$,
    \item $\phi$ extends continuously to $\overline{D}\cap U$, $\phi(p) = 0$, and
    \item $\phi(z)<0,$ for $z \in \overline{D}\cap U\setminus\{p\}.$
\end{enumerate}

If $U$ is a neighbourhood of $\overline{D}$, then $\phi$ is called a \emph{global plurisubharmonic peak function} at $p$.
\end{defn}
\begin{exmp}
    Let $D \subset \mathbb{C}^n$ be a domain and let $p \in bD$. Let $f$ be a local holomorphic peak point, i.e. there exists a neighbourhood $U$ of $p$ such that $f$ is holomorphic on $D \cap U$, extends continuously to $\overline{D} \cap U$, $f(p) = 1$, and $|f(z)| < 1$, for $z \in \overline{D} \cap U\setminus \{p\}$. Then, $\phi = \operatorname{log}|f|$ is a local plurisubharmonic peak function at $p \in bD$.
\end{exmp}

\begin{defn}
Let $D \subset \mathbb{C}^n$ be a domain and let $p \in bD$. A function $\psi$ is called a \emph{local plurisubharmonic antipeak function} at $p$ if there exists a neighbourhood $U$ of $p$ such that:
\begin{enumerate}
    \item $\psi$ is plurisubharmonic on $D \cap U$,
    \item $\psi$ extends continuously to $\overline{D}\cap U$, $\psi(p) = -\infty$, and
    \item $\psi(z)>-\infty,$ for $z \in (\overline{D}\cap U)\setminus\{p\}$.
\end{enumerate}

If $U$ is a neighbourhood of $\overline{D}$, then $\psi$ is called a \emph{global plurisubharmonic antipeak function} at $p$.
\end{defn}
\begin{exmp}
    Let $D \subset \mathbb{C}^n$ be a domain and let $p \in bD$. Define 
    \[\psi(z) = \operatorname{log}|z - p|.\]
    Then it is easy to verify that $\psi$ is a local plurisubharmonic antipeak function at $p$. Hence, every domain in $\mathbb{C}^n$ admits a local plurisubharmonic antipeak function at each boundary point of the domain.
\end{exmp}

The following lemma guarantees the existence of global plurisubharmonic peak and antipeak functions provided local ones exist simultaneously.

\begin{lem}[{\cite[Proof of Lemma 2.1.1]{Gaussier}}]\label{ex. neg peak and antipeak}
Let $D$ be a domain in $\mathbb{C}^n$ and let $p \in bD$. Assume that there exist local plurisubharmonic peak and antipeak functions $\phi$ and $\psi$ at $p$. Then, there exist global negative plurisubharmonic
peak and antipeak functions $\widetilde{\phi}$ and $\widetilde{\psi}$ at $p$ such that
\[
{\widetilde{\psi}}^{-1}(-\infty)=\{p\}.
\]
\end{lem}
The main tool used to prove the strong localization is the \emph{pluricomplex Green function}.
The pluricomplex Green function of $D$ with pole at $z \in D$ is defined by
\begin{align*}
    g_D(z,w)
    =
    \sup
    \left\{
    u(w):
    \begin{array}{l}
    u \text{ is plurisubharmonic on } D,\\
    u<0,\\
    u(\cdot)
    <
    \log|\cdot-z|+\mathcal{O}_z(1)
    \end{array}
    \right\}.
\end{align*}

The pluricomplex Green function is biholomorphically invariant. More precisely, if
\[
F:D_1 \to D_2
\]
is a biholomorphic map, then
\begin{align*}
    g_{D_2}\big(F(z),F(w)\big)
    =
    g_{D_1}(z,w),
    \qquad z,w \in D_1.
\end{align*}
We will use the following Schwarz lemma for log-subharmonic functions in the proof of Theorem \ref{1.1}. A function $u:D \to [0,\infty)$ is said to be \emph{log-subharmonic} on $D$ if $\log u$ is subharmonic on $D$.

\begin{lem}[{\cite[p.~794]{Jarnicki}}]\label{Schwarz}
Let $u$ be a log-subharmonic function on $\mathbb{D}$ such that:
\begin{enumerate}
    \item the function $\lambda \mapsto u(\lambda)/|\lambda|$ is bounded near $0$, and
    \item 
    $\limsup_{|\lambda| \to 1^-} u(\lambda)\leq 1.$
\end{enumerate}
Then,
\begin{align*}
    u(\lambda)\leq |\lambda|
    \quad \text{for all}\;\,\lambda \in \mathbb{D}.
\end{align*}
\end{lem}
\section{Strong Localization of the Kobayashi–Eisenman volume element}\label{Section 3}
In this section, we first establish the localization result for the Kobayashi–Eisenman volume element in terms of the Lempert function; see Proposition \ref{loc of kob. vol.}. We then recall a construction due to Nikolov and Fornæss \cite{Nikolai} of a bounded plurisubharmonic function in Lemma~\ref{bdd psh fun}. Using this bounded plurisubharmonic function, we derive a lower bound for the pluricomplex Green function in terms of a local plurisubharmonic peak function in Lemmas~\ref{4.2} and~\ref{4.3}. Finally, combining Proposition \ref{loc of kob. vol.}, Lemma \ref{4.2}, and Lemma \ref{4.3}, we prove a quantitative version of strong localization for the Kobayashi–Eisenman volume element near a plurisubharmonic peak point.

Let $D \subset \mathbb{C}^n$ be a domain. Recall the Lempert function 
\begin{align}
    \ell_D(z, w) &= \inf\{|r|: \exists f \in \mathcal{O}(\mathbb{D}, D), f(0) = z, f(r) = w\}.
\end{align}
We define the following auxiliary Lempert function
\begin{align}
    \tilde{\ell}_D(z, w) &= \inf\{|\alpha|: \exists \psi \in \mathcal{O}(\mathbb{B}^n, D), \psi(0) = z, \psi(\alpha) = w\}.
\end{align}
We will use the auxiliary Lempert function in the proof of Proposition \ref{loc of kob. vol.}. The following lemma shows that the Lempert function and the auxiliary Lempert function coincide.
\begin{lem}
    Let $D \subset \mathbb{C}^n$ be a domain. Then,
    \begin{align*}
        \tilde{\ell}_D(z, w) = \ell_D(z, w)
    \end{align*}
    for $z, w \in D$.
\end{lem}
\begin{proof}
    Let $z, w \in D$. Using unitary maps, we can have
    \begin{align*}
        \tilde{\ell}_D(z, w) = \inf \{|r|: \exists \psi \in \mathcal{O}(\mathbb{B}^n, D), \psi(0) = z, \psi(r, 0, \dots, 0) = w\}.
    \end{align*}
    For any $\psi \in \mathcal{O}(\mathbb{B}^n, D)$, we have a corresponding $f \in \mathcal{O}(\mathbb{D}, D)$ given by $f(\lambda) = \psi(\lambda, 0, \dots, 0)$; and for any $g \in \mathcal{O}(\mathbb{D}, D)$, we have a related $\Psi \in \mathcal{O}(\mathbb{B}^n, D)$ given by $\Psi(z) = g(z_1)$. Using these corresponding functions in the definitions of $\ell_D$ and $\ti \ell_D$, we easily obtain
    \begin{align*}
        \tilde{\ell}_D(z, w) = \ell_D(z, w).
    \end{align*}
\end{proof}
Since the Kobayashi–Eisenman volume element is monotonically decreasing with respect to domain inclusion, it is natural to seek a reverse-type inequality. The following proposition provides such an inequality with a weight depending on the Lempert function.
\begin{prop}\label{loc of kob. vol.}
    Let $D$ be a domain in $\mathbb{C}^n$ and $D_1 \subset D$ be any non-empty subdomain. Then,
    \begin{align}\label{23}
        k_D(z) \geq \inf_{w \in D \setminus D_1}\big(\ell_{D}(z, w)\big)^{2n}k_{D_1}(z) = \inf_{w \in D \cap bD_1}\big(\ell_{D}(z, w)\big)^{2n}k_{D_1}(z),
    \end{align}
    for $z \in D_1$.
\end{prop}
\begin{proof}
    Let $z \in D_1$. If $\inf_{w \in D \setminus D_1}\tilde{\ell}_D(z, w) = 0$, then \eqref{23} is always true. 

    Assume $\inf_{w \in D \setminus D_1}\tilde{\ell}_D(z, w) > 0$. Choose $s > 0$ such that
    \begin{align}\label{27a}
        \inf_{w \in D \setminus D_1}\tilde{\ell}_D(z, w) > s.
    \end{align}
    For any $\psi \in \mathcal{O}(\mathbb{B}^n, D)$ with $\psi(0) = z$, using \eqref{27a}, we have
    \begin{align*}
        \psi(w) \in D_1 \quad \text{for all } w \in s\mathbb{B}^n.
    \end{align*}
    Hence
    \begin{align*}
        k_{D_1}(z) \leq \frac{|\operatorname{det}\psi'(0)|^{-2}}{s^{2n}}.
    \end{align*}
    Here $\psi$ and $s$ are arbitrary, therefore
    \begin{align*}
        k_{D}(z) \geq \inf_{w \in D \setminus D_1}\big(\tilde{\ell}_D(z, w)\big)^{2n} k_{D_1}(z).
    \end{align*}
    Since $\tilde{\ell}_D = \ell_D$,
    \begin{align*}
         k_{D}(z) \geq \inf_{w \in D \setminus D_1}\big(\ell_{D}(z, w)\big)^{2n} k_{D_1}(z),
    \end{align*}
    for $z \in D_1$. Also, it is easy to see that
    \begin{align*}
        \inf_{w \in D\setminus D_1}\ell_{D}(z, w) = \inf_{w \in D \cap bD_1} \ell_D(z, w),
    \end{align*}
    for $z \in D_1$.
\end{proof}
Nikolov and Fornæss \cite{Nikolai} constructed bounded plurisubharmonic functions on domains whose boundary contains a plurisubharmonic peak point. We recall their construction below.
\begin{lem}\label{bdd psh fun}
    Let $D \subset \mathbb{C}^n$ be a domain and let $p \in bD$. Suppose that there exists a local plurisubharmonic peak function $\phi$ at $p \in bD$. Then, there exist a bounded plurisubharmonic function ${\theta}$ on $D$ and $r \in (0, 1)$ such that
    \begin{align}
        \theta(z) = |z - p|^2
    \end{align}
    for $z \in D \cap \mathbb{B}^n(p, r)$.
\end{lem}
\begin{proof}
    By translation, we may assume that $p = 0$.
    Using Lemma \ref{ex. neg peak and antipeak}, there exists a negative antipeak plurisubharmonic function $\hat{\theta}$ for $D$ at $0$. Then, there exists $r \in (0, 1)$ such that
\begin{align*}
    \inf_{D \setminus \mathbb{B}^n} \hat{\theta} \geq c := 1 + \sup_{D \cap \overline{\mathbb{B}^n(0, r)}} \hat{\theta}.
\end{align*}
If $D \subset \mathbb{B}^n$, then we may define $\inf_{D \setminus \mathbb{B}^n} \hat{\theta} := c$. Set 
\begin{align*}
    \tilde{\theta} = 1 + (1 - r^2) (\hat{\theta} - c),
\end{align*}
then 
\begin{align}
        \theta = \begin{cases} 
      |\cdot|^2,  & D \cap \mathbb{B}^n(0, r),\\
      \operatorname{max}\{|\cdot|^2, \tilde{\theta}\},  & D \cap \big(\mathbb{B}^n \setminus \mathbb{B}^n(0, r)\big),\\
      \tilde{\theta}, &D \setminus \mathbb{B}^n,
      \end{cases}
\end{align}
is a bounded plurisubharmonic function on $D$.
\end{proof}
The following lemma provides a lower bound for the pluricomplex Green function in terms of a local plurisubharmonic peak function. The proof is based on the construction of negative plurisubharmonic functions on the domain $D$ from a given local plurisubharmonic peak function and a bounded plurisubharmonic function on the domain $D$ obtained in Lemma~\ref{bdd psh fun}.
\begin{lem}\label{4.2}
    Let $D \subset \mathbb{C}^n$ be a domain and let $p \in bD$. Suppose that there exists a local plurisubharmonic peak function $\phi$ at $p \in bD$. Then, there exists a neighbourhood $\hat{U}$ of $p$ such that for any neighbourhood $V \subset \subset \hat{U}$ of $p$, there exists a constant $m > 0$ such that
    \begin{align}\label{unbdd lower of green}
        \inf_{w \in D \cap b\hat{U}}g_D(w, z) \geq m \phi(z),
    \end{align}
    for $z \in D \cap V$.
\end{lem}
\begin{proof}
By translation, we may assume that $p = 0$. By Lemma \ref{bdd psh fun}, there exist a bounded plurisubharmonic function $\theta$ on $D$ and a constant $r \in (0, 1)$ such that $\theta(z) = |z|^2$ on $D \cap \mathbb{B}^n(0, r)$. Choose $s \in (0, r)$ such that the local plurisubharmonic peak function $\phi$ is defined on $D \cap \mathbb{B}^n(0, s)$. Take $\hat{U} = \mathbb{B}^n(0, s)$. Choose a cut-off function $\chi \in C_c^{\infty}(-1, 1)$ such that
\begin{align*}
    \chi = 1 \, \,\text{on} \, \, (-1/2, 1/2) \quad \text{and} \quad 0 \leq \chi \leq 1.
\end{align*}
Let $w \in D \cap b\hat{U}$. For $z \in D$, define
\begin{align}
    f(z) := \chi\left(\frac{|z - w|^2}{(r - s)^2}\right)\log(|z - w|^2).
\end{align}
Then, there exists a constant $\hat{m}(r,s) > 0$ (independent of $w$) such that
\begin{align}
    \sum_{i, j = 1}^{n} \frac{\partial^2 f}{\partial z_i \partial \bar{z}_j}(z)\xi_i \bar{\xi}_j \geq -\hat{m}(r, s)|\xi|^2,
\end{align}
for $z \in D\setminus \{w\}$. Define 
\begin{align}
    \psi(z) := \frac{1}{2}f(z) + \hat{m}(r, s) \big({\theta} - \sup_{D} \theta - 1\big).
\end{align}
Here, $\psi$ is a negative plurisubharmonic function on $D$ with pole at $w$.

For any neighbourhood $V \subset \subset \hat{U}$ of $p$, we may choose large $m> 0$ such that
\begin{align*}
       \widetilde{\phi} =
       \begin{cases} 
       \operatorname{max}\{\psi, m\phi\},  & D \cap V,\\
       \psi, &D \setminus V,
      \end{cases}
\end{align*}
is a negative plurisubharmonic function on $D$. Hence
\begin{align*}
    g_D(w, z) \geq \widetilde{\phi}(z) \geq m\phi(z),
\end{align*}
for $z \in D \cap V$, where $m$ is independent of $w \in D \cap b\hat{U}$. Therefore,
\begin{align}
    \inf_{w \in D \cap b\hat{U}} g_D(w, z) \geq {m}\phi(z),
\end{align}
for $z \in D \cap V$.
\end{proof}
If we assume that the domain $D$ is bounded, then for any neighbourhood $U$ of $p \in bD$, estimate \eqref{unbdd lower of green} holds, as shown in the following lemma.
\begin{lem}\label{4.3}
    Let $D \subset \mathbb{C}^n$ be a bounded domain and let $p \in bD$. Suppose that there exists a local plurisubharmonic peak function $\phi$ at $p \in bD$, defined on a neighbourhood $U_0$ of $p$. Then, for any pair of neighbourhoods $V$ and $U$ of $p$ satisfying $V \subset \subset U \subset U_0$, there exists a constant $m > 0$ such that
    \begin{align}
        \inf_{w \in D \cap bU}g_D(w, z) \geq m \phi(z),
    \end{align}
    for $z \in D \cap V$.
\end{lem}
\begin{proof}
   By translation, we may assume that $p = 0$. Choose a cut-off function $\chi \in C_c^{\infty}(-1, 1)$ such that
\begin{align*}
    \chi = 1 \, \,\text{on} \, \, (-1/2, 1/2) \quad \text{and} \quad 0 \leq \chi \leq 1.
\end{align*}
Let $w \in D \cap bU$. For $z \in D$, define
\begin{align}
    f(z) := \chi\left({|z - w|^2}\right)\log(|z - w|^2).
\end{align}
Then, there exists a constant $\hat{m} > 0$ (independent of $w$) such that
\begin{align}
    \sum_{i, j = 1}^{n} \frac{\partial^2 f}{\partial z_i \partial \bar{z}_j}(z)\xi_i \bar{\xi}_j \geq -\hat{m}|\xi|^2,
\end{align}
for $z \in D\setminus \{w\}$. Define 
\begin{align}
    \psi(z) := \frac{1}{2}f(z) + \hat{m} \bigg(|z|^2 - \big(\operatorname{diam}(D)\big)^2 - 1\bigg).
\end{align}
Here, $\psi$ is a negative plurisubharmonic function on $D$ with pole at $w$.

For any neighbourhood $V \subset \subset U$ of $p$, we may choose
large ${m} > 0$ such that
\begin{align*}
       \widetilde{\phi} =
       \begin{cases} 
       \operatorname{max}\{\psi,{m}\phi\},  & D \cap V,\\
       \psi, &D \setminus V,
      \end{cases}
\end{align*}
is a negative plurisubharmonic function on $D$. Hence
\begin{align*}
    g_D(w, z) \geq \widetilde{\phi}(z) \geq m\phi(z),
\end{align*}
for $z \in D \cap V$, where $m$ is independent of $w \in D \cap bU$. Therefore,
\begin{align}
    \inf_{w \in D \cap bU} g_D(w, z) \geq m\phi(z),
\end{align}
for $z \in D \cap V$.    
\end{proof}
We now combine the above lemmas to prove Theorem \ref{1.1}.
\begin{proof}[Proof of Theorem \ref{1.1}]
\emph{(i)}
    Using Lemma \ref{4.2}, there exists a neighbourhood $\hat{U} \subset U$ of $p \in bD$ such that for any neighbourhood $V \subset \subset \hat{U}$ of $p$, there exist a constant $m > 0$ such that
    \begin{align}\label{46a}
        \inf_{w \in D \cap b\hat{U}}g_D(w, z) \geq m \phi(z),
    \end{align}
    for $z \in D \cap V$.
    Using Proposition \ref{loc of kob. vol.}, we have
    \begin{align}\label{47a}
    k_{D}(z) \geq \inf_{w \in D \cap b\hat{U}}\big(\ell_D(z, w)\big)^{2n} k_{D \cap \hat{U}}(z) \geq \inf_{w \in D \cap b\hat{U}}\big(\ell_D(z, w)\big)^{2n} k_{D \cap U}(z).
    \end{align}
    Using Lemma \ref{Schwarz}, we have
    \begin{align}\label{48a}
        \operatorname{exp}\big(g_D(w, z)\big) \leq \ell_D(w, z) = \ell_D(z, w),
    \end{align}
    for $w, z \in D$. From \eqref{46a}, \eqref{47a} and \eqref{48a}, we have
    \begin{align}\label{60}
        k_D(z) \geq \operatorname{exp}\big(2n m\phi(z)\big) k_{D \cap U}(z),
    \end{align}
    for $z \in D \cap V$.

If $D$ is a bounded domain, then we may use Lemma \ref{4.3} instead of Lemma \ref{4.2} to prove \eqref{60} for any pair of neighbourhoods $V$ and $U$ of $p$ satisfying $V \subset \subset U$ and $D \cap U$ is 
connected.

\emph{(ii)} To prove the strong localization of the quotient invariant, we first show that \(q_D\) and \(q_{D \cap U}\) are well defined near \(p\), i.e., \(k_D\) and \(k_{D \cap U}\) are positive near \(p\). We then use the strong localization of the Kobayashi–Eisenman volume element to obtain \eqref{38}.

Let \(U_0\) be a bounded neighbourhood of \(p\). By part \emph{(i)} of Theorem~\ref{1.1}, there exist a neighbourhood \(V_0 \subset U_0\) of \(p\) and a constant \(m_0>0\) such that
\begin{align}\label{62}
    k_D(z) \ge e^{m_0\phi(z)}\, k_{D \cap U_0}(z),
\end{align}
for \(z \in D \cap V_0\).

Since \(U_0\) is bounded, we have $
k_{D \cap U_0}(z)>0.$
Therefore, by \eqref{62}, it follows that
\begin{align}\label{63}
    k_D(z)>0
\end{align}
for \(z \in D \cap V_0\). Hence \(q_D\) is well defined on \(D \cap V_0\).

Now let \(U\) be any neighbourhood of \(p\). Then \(p \in b(D \cap U)\), and \(p\) is a local peak point for \(D \cap U\). Arguing as above, there exists a neighbourhood \(V_0'\) of \(p\) such that
\begin{align}\label{64}
    k_{D \cap U}(z)>0
\end{align}
for \(z \in D \cap U \cap V_0'\). Hence \(q_{D \cap U}\) is well defined on \(D \cap U \cap V_0'\).

Again by part \emph{(i)} of Theorem~\ref{1.1}, there exist a neighbourhood \(V \subset U\) of \(p\) and a constant \(m>0\) such that
\begin{align}\label{65}
    k_D(z) \ge e^{m\phi(z)}\, k_{D \cap U}(z),
\end{align}
for \(z \in D \cap V\).

Using the monotonicity of the Carathéodory--Eisenman volume element under domain inclusion together with \eqref{65}, we obtain
\begin{align} \label{38}
    q_{D \cap U}(z)
    =
    \frac{c_{D \cap U}(z)}{k_{D \cap U}(z)}
    \ge
    e^{m\phi(z)}
    \frac{c_D(z)}{k_D(z)}
    =
    e^{m\phi(z)} q_D(z),
\end{align}
for $z \in D \cap V \cap V_0 \cap V_0'$.

If $D$ is bounded, then both $k_D$ and $k_{D \cap U}$ are positive. Hence, we may apply part \emph{(i)} of Theorem~\ref{1.1} directly to obtain the desired result.
\end{proof}
\section{Asyptotic behaviour of the Kobayashi–Eisenman volume element}\label{Cone}
In this section, we prove the nontangential asymptotic limit of the Kobayashi–Eisenman volume element at exponentially flat infinite type boundary points of domains in $\mathbb{C}^{n + 1}$. After establishing the appropriate localization Lemma~\ref{1.3}, we apply the method of scaling to prove our result (Theorem~\ref{I_1 2}).

Let $D \subset \mathbb{C}^{n + 1}$ be a bounded smooth domain with $0 \in bD$ and let $\gamma(t) = (-t, 0)$. If $0$ is an exponentially flat boundary point, then there exists a neighbourhood $U$ of the origin, such that
\begin{align}\label{*}
    D \cap U = \{z \in U: \rho(z) < 0\},
\end{align}
where the defining function $\rho$ is defined as in \eqref{1}, and $D \cap U$ is convex. 
Since $D \cap U$ is convex, by Lemma \ref{bdd psh fun} there exist a bounded plurisubharmonic function ${\theta}$ on $D$ and a constant $r \in (0, 1)$ such that
\begin{align}\label{40}
    {\theta}(z) =  |z|^2,
\end{align}
for $z \in D \cap \mathbb{B}^{n + 1}(0, r)$. 
Let $W \subset \subset \mathbb{B}^{n + 1}(0, r) \cap U$ be a neighbourhood of origin and assume 
\[\Omega = \{z \in \mathbb{C}^n: \rho(z) < 0\},\]
then $D \cap W = \Omega \cap W$.

To prove the theorem, we follow the construction of the domain as given in \cite{ravi}.

We first cut out a portion of $D$ near the origin, depending on the point $\gamma(t)$, and denote it by $D_t^{\epsilon}$ for $\epsilon > 0$. We then localize $k_D$ to $k_{D_t^{\epsilon}}$. Next, we apply the scaling map $\Sigma$ and biholomorphic map $f$ on $D_t^{\epsilon}$ so that the resulting domain converges to $\mathbb{D} \times \mathbb{B}^n$, we refer to as the scaling lemma. Using the scaling lemma, we conclude $k_{f \circ \Sigma(D_t^{\epsilon})}(0)$ converges to $k_{\mathbb{D} \times \mathbb{B}^n}(0)$. Combining these arguments, we obtain the asymptotic behaviour of $k_D\big(\gamma(t)\big)$.

For small $\epsilon, t > 0$, define
\begin{align}
    D_t^{\epsilon} &= \left\{(z_1, \dots, z_{n+1}) \in W : \rho(z_1, z_2, \dots, z_{n+1}) < 0, \operatorname{Re}z_1 > - t^{{1}/{(1+\epsilon)^2}}\right\},
\end{align}
and $h_\epsilon : D_t^{\epsilon} \to \mathbb{D}$ is the holomorphic peak function of $D_t^{\epsilon}$ at zero, given by 
\begin{align}
    h_\epsilon(z) = \operatorname{exp}\left(-(-z_1)^{{1}/{(1+ \epsilon)}}\right). 
\end{align}
Therefore, the function $\phi_{\epsilon} : D_t^{\epsilon} \to \mathbb{D}$, defined by
\begin{align}
    \phi_{\epsilon}(z) = \operatorname{log}|h_{\epsilon}(z)| = -|z_1|^{1/(1 + \epsilon)} \operatorname{cos}\left(\frac{\operatorname{arg}(-z_1)}{1 + \epsilon}\right), \quad z \in D_t^{\epsilon},
\end{align}
is a local plurisubharmonic peak function at $0$.

Define
\begin{align*}
    d^*(t) &:= 
    \operatorname{min}\{s \in\mathbb{R}^+: se_2 -te_1 \in b\Omega\} = \sqrt{\phi^{-1}(t)}, \\
    d_1^\epsilon(t) &:=
     \operatorname{sup} \left\{|z'| : z \in W, \phi\left(\left|z_2\right|^2 + \dots + |z_{n+1}|^2\right) \leq t^{\frac{1}{(1+\epsilon)}}\right\} = \sqrt{\phi^{-1}\left(t^{\frac{1}{(1 + \epsilon)}}\right)}, \text{ and}\\
    d_2^\epsilon(t) &:=
    \operatorname{sup}\left\{|z'| : z \in W, \phi\left(\left|z_2\right|^2 + \dots + |z_{n+1}|^2\right) \leq t^{\frac{1}{(1+\epsilon)^2}}\right\} = \sqrt{\phi^{-1}\left(t^{\frac{1}{(1 + \epsilon)^2}}\right)}.
\end{align*}
Define the scaling map $\Sigma : \mathbb{C}^{n+1} \to \mathbb{C}^{n+1}$
by
\begin{align}
    \Sigma(z_1, z') = \left(\frac{z_1}{t}, \frac{z'}{d^*(t)}\right),
\end{align}
where $z' = (z_2, \dots, z_{n+1})$.
\begin{Remark}
Here, \( d^{*}(t) \), \( d_{1}^{\epsilon}(t) \), and \( d_{2}^{\epsilon}(t) \) denote the distances from \( -t e_1 \), \( -t^{\frac{1}{1+\epsilon}} e_1 \), and \( -t^{\frac{1}{(1+\epsilon)^2}} e_1 \), respectively, to \( b\Omega \) in the direction of \( e_2 \). 
By the definitions of \( d^{*}, d_1^{\epsilon}, \) and \( d_2^{\epsilon} \), we have
    \begin{align}
        d^{*}(t) \leq d_1^{\epsilon}(t) \leq d_2^{\epsilon}(t),
    \end{align}
    for all sufficiently small \( \epsilon, t > 0 \), and using \eqref{asy of psi}, we obtain
    \begin{align}
        \lim_{t \to 0^{+}} \frac{d_1^{\epsilon}(t)}{d^{*}(t)} &= (1 + \epsilon)^{\frac{1}{2m}},  \text{ and} \quad\\
        \lim_{t \to 0^{+}} \frac{d_2^{\epsilon}(t)}{d^{*}(t)} &= (1 + \epsilon)^{\frac{1}{m}} \label{asy of d_2^{epsilon}}.        
    \end{align}
\end{Remark}
We localize $k_D$ to $k_{D_t^{\epsilon}}$ in the following lemma. The proof follows the same method as in Lemma \ref{4.2}; however, in the present setting, we explicitly choose the constants to construct negative plurisubharmonic functions on the domain $D$.
\begin{lem}\label{1.3}
    Let $D \subset \mathbb{C}^{n + 1}$ be a domain and let $0 \in bD$. Assume $0$ is an exponentially flat boundary point and $\gamma(t) = (-t, 0)$, then, for $\epsilon > 0$,
    \begin{align}
        \lim_{t \to 0^{+}}\frac{k_D\big(\gamma(t)\big)}{k_{D_t^{\epsilon}}\big(\gamma(t)\big)} = 1.
    \end{align}
\end{lem}
\begin{proof}
Choose a cut-off function $\chi \in C_c^\infty(-1, 1)$ such that 
\begin{align*}
    \chi = 1 \text{ on } (-1/2, 1/2) \quad \text{and} \quad 0 \leq \chi \leq 1. 
\end{align*}
Let $\epsilon$ and $t$ be sufficiently small positive real numbers, and let $w \in 
D \cap b{D}_t^{\epsilon}$. Define
\begin{align}
    f_t^{\epsilon}(z) = \chi\left(\frac{|z - w|^2}{t^{\frac{\epsilon}{4(1 + \epsilon)^2}}}\right)\operatorname{log}(|z - w|^2).
\end{align}
Let $z \in \{v \in \mathbb{C}^{n + 1}: {t^{\frac{\epsilon}{4(1 + \epsilon)^2}}}/2 \leq |v - w|^2 \leq  {t^{\frac{\epsilon}{4(1 + \epsilon)^2}}}\}$, for $\xi \in \mathbb{C}^{n + 1}$, consider
\begin{align*}
    \sum_{i, j = 1}^{n + 1}\frac{\partial^2 f_t^{\epsilon}}{\partial z_i \partial \bar{z}_j}(z) \xi_i  \bar{\xi}_j &\geq   \sum_{i, j = 1}^{n + 1}\frac{\partial^2}{\partial z_i \partial \bar{z}_j} \chi\left(\frac{|z - w|^2 }{t^{\frac{\epsilon}{4 (1 + \epsilon)^2}}}\right) \operatorname{log}(| z - w |^2) \xi_i \bar{\xi}_j \\
    &+ \sum_{i, j = 1}^{n + 1} \frac{\partial}{\partial z_i} \chi\left(\frac{| z - w|^2 }{t^{\frac{\epsilon}{4 (1 + \epsilon)^2}}}\right) \frac{\partial}{\partial \bar{z}_j} \operatorname{log}(| z - w| ^2)\xi_i \bar{\xi}_j \\
    &+ \sum_{i, j = 1}^{n+ 1}\frac{\partial}{\partial \bar{z}_j} \chi\left(\frac{|z - w|^2 }{t^{\frac{\epsilon}{4 (1 + \epsilon)^2}}}\right) \frac{\partial}{\partial z_i} \operatorname{log}(|z - w| ^2)\xi_i\bar{\xi}_j\\
    &\geq \frac{M|\xi|^2}{t^{\frac{\epsilon}{2(1 + \epsilon)^2}}}\operatorname{log}(|z - w|^2) \\
    &+ \frac{2}{ {t^{\frac{\epsilon}{4(1 + \epsilon)^2}}|z -w|^2}}\chi'\left(\frac{|z - w|^2}{ {t^{\frac{\epsilon}{4(1 + \epsilon)^2}}}}\right)\bigg|\sum_{i = 1}^{n + 1}(\bar{z_i} - \bar{w}_i) \xi_i\bigg|^2\\
    &\geq  \frac{M|\xi|^2}{t^{\frac{\epsilon}{2(1 + \epsilon)^2}}}\operatorname{log}(|z - w|^2) - \frac{M}{ {t^{\frac{\epsilon}{4(1 + \epsilon)^2}}|z -w|^2}} \bigg|\sum_{i = 1}^{n + 1}(\bar{z_i} - \bar{w}_i) \xi_i\bigg|^2\\
    &\geq \frac{M|\xi|^2}{{t^{\frac{\epsilon}{2(1 + \epsilon)^2}}}} \operatorname{log}({t}^{\frac{\epsilon}{4(1 + \epsilon)^2}}/2) - \frac{M|\xi|^2}{{t^{\frac{\epsilon}{4(1 + \epsilon)^2}}}}\\
    &\geq \frac{M \epsilon |\xi|^2}{{4(1 + \epsilon)^2 t^{\frac{\epsilon}{2(1 + \epsilon)^2}}}}\operatorname{log}t -  \frac{M|\xi|^2}{{t^{\frac{\epsilon}{2(1 + \epsilon)^2}}}} \operatorname{log}(2) - \frac{M|\xi|^2}{t^{\frac{\epsilon}{4(1 + \epsilon)^2}}}\\
    &\geq - \frac{M|\xi|^2}{t^{\frac{\epsilon}{2(1 + \epsilon)^2}}} \left( 1 - \frac{\epsilon}{4(1 + \epsilon)^2} \operatorname{log}t\right).
\end{align*}
 The constant \(M>0\) appearing above may vary from step to step and is independent of sufficiently small $t > 0$ and $\epsilon > 0$. Hence,
\begin{align}
    \sum_{i, j = 1}^{n + 1} \frac{\partial^2 f_t^{\epsilon}}{\partial z_i \partial \bar{z}_j}(z)\xi_i \bar{\xi}_j \geq -M \frac{\left(1 - \frac{\epsilon}{4 (1 + \epsilon)^2}\operatorname{log}t\right)}{t^{\frac{\epsilon}{2(1 + \epsilon)^2}}} |\xi|^2 =: -m_t^{\epsilon}|\xi|^2,
\end{align}
for $\xi \in \mathbb{C}^{n + 1}$ and $z \in \mathbb{C}^{n + 1}$. Define 
\begin{align}
    \psi_t^{\epsilon} := \frac{1}{2}f_t^{\epsilon} + m_t^{\epsilon} \bigg(\theta - \sup_{D} {\theta} - 1\bigg),
\end{align}
where $\theta$ is the bounded plurisubharmonic function on $D$, defined in \eqref{40}.
Then, there exists $t_0(\epsilon) > 0$ such that $\psi_t^{\epsilon}$ is a negative plurisubharmonic function on $D$ with pole at $w$ for $t \in \big(0, t_0(\epsilon)\big)$.

Let $\widetilde{m}_t^{\epsilon} > 0$, chosen later, so that 
\begin{align*}
       \widetilde{\phi}_t^{\epsilon} := \begin{cases} 
      \operatorname{max}\{\psi_t^{\epsilon}, \widetilde{m}_t^{\epsilon}\phi_{\epsilon}\},  & D \cap \big\{z \in W: |z_1| < t^{1/(1 + \epsilon)}\big\},\\
      \psi_t^{\epsilon},  & D \setminus \big\{z \in W: |z_1| < t^{1/(1 + \epsilon)}\big\},
      \end{cases}
\end{align*}
is a negative plurisubharmonic function on $D$. Since $\phi_{\epsilon}$ is a plurisubharmonic function on $D_t^{\epsilon} \supset D \cap \{z \in W: |z_1| < t^{1/(1 + \epsilon)}\}$ and $\psi_t^{\epsilon}$ is a plurisubharmonic function on $D$, if we can choose $\widetilde{m}_t^{\epsilon} > 0$ so that $\psi_t^{\epsilon} \geq m_t^{\epsilon} \phi_{\epsilon}$ on $D \cap b\{z \in W : |z_1| < t^{1/ (1 + \epsilon)}\} = D \cap \{z \in W: |z_1| = t^{1/(1 + \epsilon)}\}$, then $\widetilde{\phi}_t^{\epsilon}$ is a plurisubharmonic function on $D$.

Take 
\begin{align}
    \widetilde{m}_t^{\epsilon} = \frac{(1 + \epsilon)^2\operatorname{log}2 - (1 + \epsilon)^2\left(\inf_{D}{\theta} - \sup_{D}{\theta} - 1\right)m_t^{\epsilon} - \operatorname{log}t}{(1 + \epsilon)^2 c_0^{\epsilon}t^{1/(1 + \epsilon)^2}},
\end{align}
where $c_0^{\epsilon} = \operatorname{cos}\left(\frac{\pi}{2(1 + \epsilon)}\right)$.
Let $z \in D \cap \{z \in W: |z_1| = t^{1/(1 + \epsilon)}\}$, consider
\begin{align*}
    \psi_t^{\epsilon}(z) &= \chi\left(\frac{|z - w|^2}{t^{\frac{\epsilon}{4(1+ \epsilon)^2}}}\right) \operatorname{log}\big(|z - w|\big) + m_t^{\epsilon} \big(\hat{\theta}(z) - \operatorname{sup} \hat{\theta} - 1 \big)\\
    &\geq \operatorname{log}\big(|z - w|\big) + m_t^{\epsilon}\big(\inf_{D} {\theta} - \sup_{D} {\theta} - 1\big) \quad\\
    &\geq \operatorname{log}\Big(t^{\frac{1}{(1 + \epsilon)^2}}\big(1 - t^{\frac{\epsilon}{(1 + \epsilon)^2}}\big)\Big) + m_t^{\epsilon} \big(\inf_D\hat{\theta} - \operatorname{sup} \hat{\theta} - 1\big)\quad (\text{as } w \in D \cap bD_t^{\epsilon})\\
    &\geq \frac{1}{(1 + \epsilon)^2} \operatorname{log}t + \operatorname{log}\left(\frac{1}{2}\right) + m_t^{\epsilon}\big(\inf_{D} \hat{\theta} - \sup_{D} \hat{\theta} - 1\big)\\
    &= -\widetilde{m}_t^{\epsilon} c_0^{\epsilon}t^{1/(1 + \epsilon)^2} \\
    &\geq - \widetilde{m}_t^{\epsilon} |z_1|^{1/(1 + \epsilon)} \operatorname{cos}\left(\frac{\operatorname{arg}(-z_1)}{(1 + \epsilon)}\right) \quad (\text{as } \operatorname{Re}z_1 \geq 0)\\
    &= \widetilde{m}^{\epsilon}_t \phi_{\epsilon}(z).
\end{align*}
Therefore, $\widetilde{\phi}_t^{\epsilon}$ is a negative plurisubharmonic function on $D$ with pole at $w$. Hence, the pluricomplex Green function of $D$ with pole at $w$
\begin{align*}
    g_D\big(w, \gamma(t)\big) &= \sup\big\{u\big(\gamma(t)\big) : u \text{ is plurisubharmonic on}\, D , u < 0, u < \operatorname{log}|\cdot - w| + \mathcal{O}_w(1)\big\}\\
    &\geq \widetilde{\phi}_t^{\epsilon}\big(\gamma(t)\big)\\
    &\geq \widetilde{m}_t^{\epsilon}\phi_{\epsilon}\big(\gamma(t)\big)\\
    &= -\widetilde{m}_t^{\epsilon} t^{1/(1 + \epsilon)},
\end{align*}
for $w \in D \cap bD_t^{\epsilon}$. Therefore,
\begin{align}\label{88}
    \inf_{w \in D \cap bD_t^{\epsilon}} g_D\big(w, \gamma(t)\big) \geq - \widetilde{m}_t^{\epsilon} t^{1/(1 + \epsilon)},
\end{align}
for $0 < t < t_0(\epsilon)$. Using Lemma \ref{Schwarz}, we have
\begin{align}\label{48a'}
        \operatorname{exp}\left(g_D\big(w, \gamma(t)\big)\right) \leq \ell_D\big(w, \gamma(t)\big) = \ell_D\big(\gamma(t), w\big).
\end{align}
Using Proposition \ref{loc of kob. vol.}, we have
    \begin{align}\label{47a'}
    k_{D}\big(\gamma(t)\big) \geq \inf_{w \in D \cap bD_t^{\epsilon}}\left(\ell_D\big(\gamma(t), w \big)\right)^{2n} k_{D_t^{\epsilon}}\big(\gamma(t)\big),
    \end{align}
    for $0 < t < t_0(\epsilon)$. From \eqref{88}, \eqref{48a'} and \eqref{47a'}, we conclude
    \begin{align*}
        k_D\big(\gamma(t)\big) \geq \operatorname{exp}\big(- \widetilde{m}_t^{\epsilon} t^{1/(1 + \epsilon)}\big) k_{D_t^{\epsilon}}\big(\gamma(t)\big),
    \end{align*}
    for $0 < t < t_0(\epsilon)$.
Therefore,
\begin{align}
    1 \geq \frac{k_{D}\big(\gamma(t)\big)}{k_{D_t^{\epsilon}}\big(\gamma(t)\big)} \geq \operatorname{exp}\big(-\widetilde{m}_t^{\epsilon} t^{1/(1 + \epsilon)}\big).
\end{align}
Since $\operatorname{exp}\big(-\widetilde{m}_t^{\epsilon} t^{1/(1 + \epsilon)}\big) \to 1$ as $t \to 0^{+}$, 
\begin{align}
    \lim_{t \to 0^{+}} \frac{k_{D}\big(\gamma(t)\big)}{k_{D_t^{\epsilon}}\big(\gamma(t)\big)} = 1.
\end{align}
\end{proof}
Now we recall the scaling lemma, proved in \cite{ravi}, to complete the proof of Theorem \ref{I_1 2}.

Let $f : \{z \in \mathbb{C}^{n + 1}: \operatorname{Re}z_1 < 0\} \to \mathbb{D} \times \mathbb{C}^n$ be the biholomorphism given by 
\[f(z_1, z') = \left(\frac{1 + z_1}{1 - z_1}, z' \right).\]
In the following scaling lemma, the limiting domain of $f \circ \Sigma (D_t^\epsilon)$ is $\mathbb{D} \times \mathbb{B}^n$. 
\begin{lem}\cite[Lemma 3.4]{ravi}\label{ScalingLemma2}
    Let $D \subset \mathbb{C}^{n + 1}$ be as in \eqref{*}. Then, for every $\epsilon, \delta > 0,$ there exists $t_0(\delta, \epsilon) > 0$ such that
    \begin{align}\label{inclusions 2}
        (1-\delta)\left(\mathbb{D}\times \mathbb{B}^n\right) \subset f \circ \Sigma \left(D_t^\epsilon\right) \subset
        \mathbb{D} \times \mathbb{B}^n\big(0, {d}_2^\epsilon(t)/d^*(t)\big),
    \end{align}
for each $ 0 < t < t_0(\delta, \epsilon
)$. 
\end{lem}
We use the localization Lemma \ref{1.3} and the scaling Lemma \ref{ScalingLemma2} to prove the asymptotic behaviour of the Kobayasi--Eisenman volume element at exponentially flat boundary points of domains in $\mathbb{C}^{n + 1}$.
\begin{thm}\label{I_1 2}
     Let $D \subset \mathbb{C}^{n + 1}$ and $W \subset \mathbb{C}^{n + 1}$ be as defined at the beginning of Section \ref{Cone}. Further, assume that $\gamma(t) = (-t, 0)$. Then,
     \begin{align}
       \lim_{t \to 0^{+}} \frac{k_{D}\big(\gamma(t)\big)}{t^{-2} d^{*}(t)^{-2n}} = \frac{1}{4}.
   \end{align}
\end{thm}
\begin{proof}
 Let $t > 0$ be small, we have
    \begin{align}\label{24}
        k_{D_t^{\epsilon}}\big(\gamma(t)\big) = {(2t)^{-2}d^*(t)^{-2n} k_{f \circ \Sigma (D_t^{\epsilon})}(0)}.
    \end{align}
    Let $\epsilon, \delta > 0$. From Lemma \ref{ScalingLemma2}, there exists $t_{0}(\delta, \epsilon) > 0$ such that
    \begin{align*}
        k_{f\circ \Sigma(D_t^{\epsilon})}(0) &\leq k_{(1-\delta)(\mathbb{D} \times \mathbb{B}^n)}(0)\\
        &= (1 - \delta)^{-2(n + 1)} k_{\mathbb{D} \times \mathbb{B}^n}(0),
    \end{align*}
    which implies
    \begin{align}\label{136}
        \frac{k_{f\circ \Sigma(D_t^{\epsilon})}(0)}{k_{\mathbb{D} \times \mathbb{B}^n}(0)} \leq (1 - \delta)^{-2(n + 1)},
    \end{align}
    for each $0 < t < t_0(\delta, \epsilon)$.

    Again using Lemma \ref{ScalingLemma2}, we have
    \begin{align*}
        k_{f\circ \Sigma(D_t^{\epsilon})}(0) &\geq k_{\mathbb{D} \times \mathbb{B}^n(0, d_2^{\epsilon}(t)/d^{*}(t))}{(0)}\\
        &= \left(\frac{d^{*}(t)}{d_2^{\epsilon}(t)}\right)^{2n}k_{\mathbb{D} \times \mathbb{B}^n}{(0)},
    \end{align*}
    which implies
    \begin{align}\label{138}
        \frac{k_{f\circ \Sigma(D_t^{\epsilon})}(0)}{k_{\mathbb{D} \times \mathbb{B}^n}{(0)}} \geq \left(\frac{d^*(t)}{d_2^{\epsilon}(t)}\right)^{2n},
    \end{align}
    for each $0 < t < t_0(\delta, \epsilon)$.

    Now 
    \begin{align}\label{139}
        \frac{(2t)^2 d^*(t)^{2n}{k_{D}\big(\gamma(t)\big)}}{k_{\mathbb{D} \times \mathbb{B}^n}(0)} &=\frac{{{k_{f\circ \Sigma(D_t^{\epsilon})}(0)}}}{k_{\mathbb{D} \times \mathbb{B}^n}(0)} \cdot \frac{(2t)^{2}d^*(t)^{2n}k_{D}\big(\gamma(t)\big)}{{{k_{f\circ \Sigma(D_t^{\epsilon})}(0)}}}\nonumber\\
        &=\frac{{{k_{f\circ \Sigma(D_t^{\epsilon})}(0)}}}{k_{\mathbb{D} \times \mathbb{B}^n}(0)} \cdot \frac{k_{D}\big(\gamma(t)\big)}{{{k_{D_t^{\epsilon}}\big(\gamma(t)\big)}}} \, \, \, \, \, \, \text{     
        (from \eqref{24})}.
    \end{align}
    By using \eqref{136}, \eqref{139} and Lemma \ref{1.3}, we conclude
    \begin{align*}
        \limsup_{t \to 0^+}\frac{(2t)^2 d^*(t)^{2n}{k_{D}\big(\gamma(t)\big)}}{k_{\mathbb{D} \times \mathbb{B}^n}(0)} \leq {(1 - \delta)^{-2(n + 1)}},
    \end{align*}
    for each $\delta > 0$. Hence
    \begin{align}\label{29}
        \limsup_{t \to 0^+}\frac{(2t)^2 d^*(t)^{2n}{k_{D}\big(\gamma(t)\big)}}{k_{\mathbb{D} \times \mathbb{B}^n}(0)} \leq {1}.
    \end{align}
    By using \eqref{138}, \eqref{139}, \eqref{asy of d_2^{epsilon}} and Lemma \ref{1.3}, we conclude
    \begin{equation*}
        \liminf_{t \to 0^{+}} \frac{(2t)^2 d^*(t)^{2n}{k_{D}\big(\gamma(t)\big)}}{k_{\mathbb{D} \times \mathbb{B}^n}(0)} \geq {(1 + \epsilon)^{-2n/m}},
    \end{equation*}
    for each $\epsilon > 0$. Hence
    \begin{align}\label{31}
        \liminf_{t \to 0^{+}} \frac{(2t)^2 d^*(t)^{2n}{k_{D}\big(\gamma(t)\big)}}{k_{\mathbb{D} \times \mathbb{B}^n}(0)} \geq 1.
    \end{align}
    From \eqref{29} and \eqref{31}, we have
    \begin{align}\label{116}
         \lim_{t \to 0^+} \frac{(2t)^2 d^*(t)^{2n}{k_{D}\big(\gamma(t)\big)}}{k_{\mathbb{D} \times \mathbb{B}^n}(0)} = 1.
    \end{align}
   Since $k_{\mathbb{D} \times \mathbb{B}^n}(0) = 1$ by Corollary \ref{k for poly ball}, we conclude
   \begin{align}
       \lim_{t \to 0^{+}} \frac{k_{D}\big(\gamma(t)\big)}{t^{-2} d^{*}(t)^{-2n}} = \frac{1}{4}.
   \end{align}
\end{proof}
\begin{Remark}
    Let $\gamma(t)$ be an $(\alpha, N)$-cone type stream approaching $0$. For each $t>0$, we perform a change of coordinates with respect to the point $\gamma(t)$, as in \cite[Section 3]{ravi}, so that $\gamma(t)$ lies on the negative $\operatorname{Re}z_1$-axis and $0$ lies on the boundary of $D$ in the new coordinates.

    We then follow the same method as in Lemma \ref{1.3} to establish the localization lemma for the Kobayashi–Eisenman volume element in these new coordinates. Since the scaling lemma \cite[Lemma 3.4]{ravi} remains valid in the new coordinates, and the Kobayashi–Eisenman volume element behaves well under biholomorphic maps, we may apply the localization and the scaling method, as in Theorem \ref{I_1 2}, to obtain the asymptotic behaviour of the Kobayashi–Eisenman volume element along an $(\alpha, N)$-cone type stream $\gamma(t)$ approaching $0$. This yields Theorem \ref{kob vol elem}.
\end{Remark}
\section*{Acknowledgements}
The first author is partially supported by the NSFC grant W2431006. The first author would like to thank Subhajit Ghosh for bringing the Hadamard inequality to his attention, which was used in the proof of Lemma~2.5. The second author thanks Diganta Borah for his invaluable suggestions.
\def\MR#1{\relax\ifhmode\unskip\spacefactor3000 \space\fi%
  \href{http://www.ams.org/mathscinet-getitem?mr=#1}{MR#1}}
 

\begin{bibdiv}
\begin{biblist}

\bib{Borah-Kar}{article}{
   author={Borah, D.},
   author={Kar, D.},
   title={Boundary behavior of the Carath\'eodory and Kobayashi-Eisenman
   volume elements},
   journal={Illinois J. Math.},
   volume={64},
   date={2020},
   number={2},
   pages={151--168},
   issn={0019-2082},
   review={\MR{4092953}},
   doi={10.1215/00192082-8303461},
}

\bib{Car}{article}{
   author={Carath\'{e}odory, C.},
   title={\"{U}ber die Abbildungen, die durch Systeme von analytischen
   Funktionen von mehreren Ver\"{a}nderlichen erzeugt werden},
   language={German},
   journal={Math. Z.},
   volume={34},
   date={1932},
   number={1},
   pages={758--792},
   issn={0025-5874},
   review={\MR{1545283}},
   doi={10.1007/BF01180619},
}

\bib{Cheung}{article}{
   author={Cheung, W. S.},
   author={Wong, B.},
   title={An integral inequality of an intrinsic measure on bounded domains
   in ${\bf C}^n$},
   journal={Rocky Mountain J. Math.},
   volume={22},
   date={1992},
   number={3},
   pages={825--836},
   issn={0035-7596},
   review={\MR{1183690}},
   doi={10.1216/rmjm/1181072698},
}

\bib{D'Angelo 1982}{article}{
   author={D'Angelo, J. P.},
   title={Real hypersurfaces, orders of contact, and applications},
   journal={Ann. of Math. (2)},
   volume={115},
   date={1982},
   number={3},
   pages={615--637},
   issn={0003-486X},
   review={\MR{657241}},
}

\bib{Dektyarev}{article}{
   author={Dektyarev, I. M.},
   title={Criterion for the equivalence of hyperbolic manifolds},
   language={Russian},
   journal={Funktsional. Anal. i Prilozhen.},
   volume={15},
   date={1981},
   number={4},
   pages={73--74},
   issn={0374-1990},
   review={\MR{0639204}},
}

\bib{Eisenman}{book}{
   author={Eisenman, D. A.},
   title={Intrinsic measures on complex manifolds and holomorphic mappings},
   series={Memoirs of the American Mathematical Society},
   volume={No. 96},
   publisher={American Mathematical Society, Providence, RI},
   date={1970},
   pages={i+80},
   review={\MR{0259165}},
}

\bib{Nikolai}{article}{
   author={Forn\ae ss, J. E.},
   author={Nikolov, N.},
   title={Strong localization of invariant metrics},
   journal={Math. Ann.},
   volume={383},
   date={2022},
   number={1-2},
   pages={353--360},
   issn={0025-5831},
   review={\MR{4444123}},
   doi={10.1007/s00208-021-02201-x},
}

\bib{Gaussier}{article}{
   author={Gaussier, H.},
   title={Tautness and complete hyperbolicity of domains in ${\bf C}^n$},
   journal={Proc. Amer. Math. Soc.},
   volume={127},
   date={1999},
   number={1},
   pages={105--116},
   issn={0002-9939},
   review={\MR{1458872}},
   doi={10.1090/S0002-9939-99-04492-5},
}

\bib{Graham}{article}{
   author={Graham, I.},
   author={Wu, H.},
   title={Characterizations of the unit ball $B^n$ in complex Euclidean
   space},
   journal={Math. Z.},
   volume={189},
   date={1985},
   number={4},
   pages={449--456},
   issn={0025-5874},
   review={\MR{0786275}},
   doi={10.1007/BF01168151},
}

\bib{Graham and Wu}{article}{
   author={Graham, I.},
   author={Wu, H.},
   title={Some remarks on the intrinsic measures of Eisenman},
   journal={Trans. Amer. Math. Soc.},
   volume={288},
   date={1985},
   number={2},
   pages={625--660},
   issn={0002-9947},
   review={\MR{0776396}},
   doi={10.2307/1999956},
}

\bib{Green-Krantz}{article}{
   author={Greene, R. E.},
   author={Krantz, S. G.},
   title={Characterizations of certain weakly pseudoconvex domains with
   noncompact automorphism groups},
   conference={
      title={Complex analysis},
      address={University Park, Pa.},
      date={1986},
   },
   book={
      series={Lecture Notes in Math.},
      volume={1268},
      publisher={Springer, Berlin},
   },
   isbn={3-540-18094-X},
   date={1987},
   pages={121--157},
   review={\MR{0907058}},
   doi={10.1007/BFb0097301},
}

\bib{ravi}{article}{
   author={Jaiswal, R. S.},
   title={Asymptotic behavior of the Bergman metric at infinite type points},
   journal={Bull. Lond. Math. Soc.},
   volume={57},
   date={2025},
   number={8},
   pages={2372--2394},
   issn={0024-6093},
   review={\MR{4946382}},
   doi={10.1112/blms.70100},
}

\bib{Jarnicki}{book}{
   author={Jarnicki, M.},
   author={Pflug, P.},
   title={Invariant distances and metrics in complex analysis},
   series={De Gruyter Expositions in Mathematics},
   volume={9},
   publisher={Walter de Gruyter \& Co., Berlin},
   date={1993},
   pages={xii+408},
   isbn={3-11-013251-6},
   review={\MR{1242120}},
}

\bib{Krantz book}{book}{
   author={Krantz, S. G.},
   title={Function theory of several complex variables},
   note={Reprint of the 1992 edition},
   publisher={AMS Chelsea Publishing, Providence, RI},
   date={2001},
   pages={xvi+564},
   isbn={0-8218-2724-3},
   review={\MR{1846625}},
}

\bib{Daowei}{article}{
   author={Ma, D.},
   title={Boundary behavior of invariant metrics and volume forms on
   strongly pseudoconvex domains},
   journal={Duke Math. J.},
   volume={63},
   date={1991},
   number={3},
   pages={673--697},
   issn={0012-7094},
   review={\MR{1121150}},
   doi={10.1215/S0012-7094-91-06328-3},
}

\bib{Nikolov 2018}{article}{
   author={Nikolov, N.},
   title={Behavior of the squeezing function near h-extendible boundary
   points},
   journal={Proc. Amer. Math. Soc.},
   volume={146},
   date={2018},
   number={8},
   pages={3455--3457},
   issn={0002-9939},
   review={\MR{3803670}},
   doi={10.1090/proc/14049},
}

\bib{Ro}{article}{
   author={Rosay, J.-P.},
   title={Sur une caract\'erisation de la boule parmi les domaines de ${\bf
   C}\sp{n}$\ par son groupe d'automorphismes},
   language={French, with English summary},
   journal={Ann. Inst. Fourier (Grenoble)},
   volume={29},
   date={1979},
   number={4},
   pages={ix, 91--97},
   issn={0373-0956},
   review={\MR{0558590}},
   doi={10.5802/aif.768},
}

\bib{Wong}{article}{
   author={Wong, B.},
   title={Characterization of the unit ball in ${\bf C}\sp{n}$ by its
   automorphism group},
   journal={Invent. Math.},
   volume={41},
   date={1977},
   number={3},
   pages={253--257},
   issn={0020-9910},
   review={\MR{0492401}},
   doi={10.1007/BF01403050},
}

\end{biblist}
\end{bibdiv}
\end{document}